\documentclass[a4paper,11pt]{amsart}

\newfont{\cyr}{wncyr10 scaled 1100}

\usepackage[left=2.7cm,right=2.7cm,top=3.5cm,bottom=3cm]{geometry}
\usepackage{amsthm,amssymb,amsmath,amsfonts,mathrsfs,amscd}
\usepackage[latin1]{inputenc}
\usepackage[all]{xy}
\usepackage{latexsym}
\usepackage{longtable}

\theoremstyle{plain}
\newtheorem{theorem}{Theorem}[section]

\newtheorem{lemma}[theorem]{Lemma}
\newtheorem{proposition}[theorem]{Proposition}

\theoremstyle{definition}

\newtheorem{examplewr}[theorem]{Example}
\newtheorem{ass}[theorem]{Assumption}

\theoremstyle{remark}
\newtheorem{obswr}[theorem]{Observation}
\newtheorem{remarkwr}[theorem]{Remark}

\newenvironment{remark}{\begin{remarkwr}\begin{upshape}}{\end{upshape}\end{remarkwr}}

\DeclareMathOperator{\an}{an}

\DeclareMathOperator{\alg}{alg}

\DeclareMathOperator{\dR}{dR}

\DeclareMathOperator{\oc}{oc}
\DeclareMathOperator{\ur}{u-r}

\DeclareMathOperator{\sym}{\mathrm{sym}}

\DeclareMathOperator{\Ind}{Ind}

\newcommand{\Res}{{\rm Res}}

\newcommand{\ka}{{\kappa}}

\newcommand{\hg}{{\mathbf{g}}}

\newcommand{\testf}{\breve{f}}
\newcommand{\testg}{\breve{g}} 
\newcommand{\testh}{\breve{h}}
\newcommand{\testgamma}{\breve{\gamma}}

\newcommand{\htg}{\breve{\mathbf{g}}}

\newcommand{\fc}{{\mathfrak{c}}}

\newcommand{\cW}{\mathcal W}

\newcommand{\cE}{\mathcal E}

\newcommand{\Q}{\mathbb{Q}}

\newcommand{\Z}{\mathbb{Z}}

\newcommand{\F}{\mathbb{F}}

\newcommand{\C}{\mathbb{C}}

\newcommand{\Gal}{\mathrm{Gal\,}}
\newcommand{\GL}{\mathrm{GL}}
\newcommand{\Cl}{\mathrm{Cl}}

\newcommand{\Fr}{\mathrm{Fr}}

\newcommand{\ord}{{\mathrm{ord}}}
\newfont{\gotip}{eufb10 at 12pt}

\newcommand{\cO}{{\mathcal O}}

\newcommand{\lra}{\longrightarrow}

\DeclareMathOperator{\Hom}{Hom}

\def\lra{{\longrightarrow}}

\def\GL{{\bf GL}}

\def\rp1{r\!\!+\!\!1}

\newcommand{\fhr}{{\mathfrak f_{\mathrm{HR}}}}
\newcommand{\fbdp}{{\mathfrak f_{\mathrm{BDP}}}}

\newcommand{\fpet}{{\mathfrak f_{\mathrm{Pet}}}}
\newcommand{\ehr}{{\mathfrak e_{\mathrm{HR}}}}
\newcommand{\ebdp}{{\mathfrak e_{\mathrm{BDP}}}}
\newcommand{\ekatz}{{\mathfrak e_{K}}}

\newcommand{\cEul}{\mathcal Eul}

\newcommand{\chiNN}{{\chi_{_N}}}

\include{thebibliography}

\begin{document}

\title[Stark points and Hida-Rankin $p$-adic $L$-function]
{Stark points and Hida-Rankin $p$-adic $L$-function}

\author{Daniele Casazza and Victor Rotger}

\begin{abstract}
This article is devoted to the {\em elliptic Stark conjecture} formulated in \cite{DLR}, which proposes a formula for the transcendental part of a $p$-adic avatar of the leading term at $s=1$ of the Hasse-Weil-Artin $L$-series $L(E,\varrho_1\otimes \varrho_2,s)$ of an elliptic curve $E/\Q$ twisted by the tensor product $\varrho_1\otimes \varrho_2$ of two odd $2$-dimensional Artin representations, when the order of vanishing is two. The main ingredient of this formula is a $2\times 2$ $p$-adic regulator involving the $p$-adic formal group logarithm of suitable Stark points on $E$. This conjecture was proved in \cite{DLR} in the setting where $\varrho_1$ and $\varrho_2$ are induced from characters of the same imaginary quadratic field $K$. In this note we prove a refinement of this result, that was discovered experimentally in \cite[Remark 3.4]{DLR} in a few examples. Namely, we are able to determine the algebraic constant up to which the main theorem of \cite{DLR} holds in a particular setting where the Hida-Rankin $p$-adic $L$-function associated to a pair of Hida families can be exploited to provide an alternative proof of the same result. This constant encodes local and global invariants of both $E$ and $K$.
\end{abstract}

\address{D. C.: Universit\'{e} de Bordeaux, Bordeaux, France; Universitat Polit\'{e}cnica de Catalunya, Barcelona, Spain}
\email{daniele.casazza@u-bordeaux.fr}
\address{V. R.: Departament de Matem\`{a}tica Aplicada II, Universitat Polit\`{e}cnica de Catalunya, C. Jordi Girona 1-3, 08034 Barcelona, Spain}
\email{victor.rotger@upc.edu}


\maketitle

\tableofcontents

\section{Introduction}

Let $E/\Q$ be an elliptic curve of conductor $N_E$ and let $f\in S_2(N_E)$ denote the eigenform associated to it by modularity. Let in addition
\begin{equation*}
	\varrho : G_\Q \lra \GL(V_\varrho) = \GL_n(L)
\end{equation*}
be an Artin representation with values in a finite extension $L/\Q$ and factoring through the Galois group $\Gal(H/\Q)$ of a finite Galois extension $H/\Q$. Here $V_\varrho$ is the $L[G_\Q]$-module underlying the Galois representation $\varrho$. 

We define the $\varrho$-isotypical component of the Mordell-Weil group of $E(H)$ as
$$
E(H)^\varrho = \Hom_{\Gal(H/\Q)}(V_\varrho, E(H) \otimes L),
$$
and set 
$$
 r(E,\varrho) := \dim_L E(H)^\varrho.
$$

Let $L(E,\varrho,s)$ denote the Hasse-Weil-Artin $L$-series  of the twist of $E$ by $\varrho$. This $L$-function is expected to admit analytic continuation to the whole complex plane and to satisfy a functional equation relating the values at $s$ and $2-s$, although this is only known in a few cases, including the ones we consider in this note. Assuming these properties, we may define the {\em analytic rank} of the pair $(E,\varrho)$ as
$$
r_{\an}(E, \varrho) := \ord_{s=1} L(E,\varrho,s).
$$

The equivariant refinement of the Birch and Swinnerton-Dyer conjecture (cf.\,\cite{Dok},\cite{Roh}) for the twist of $E$ by $\varrho$ predicts that
\begin{equation} \label{equivariantBSD}
r(E,\varrho) \stackrel{?}{=} r_{\an}(E, \varrho).
\end{equation}

The equality \eqref{equivariantBSD} is known to be true in rather few cases, always under the assumption that $r_{\an}(E,\varrho) = 0$ or $r_{\an}(E,\varrho) = 1$. We refer to \cite{BDR2} and \cite{DR2} for the latest developments in this direction. In particular one is often at a loss to construct non-trivial points in $E(H)^\varrho$.

In the recent work  \cite{DLR}, Darmon, Lauder and Rotger propose a new approach for computing (linear combinations of logarithms of) non-zero elements in $E(H)^\varrho$ for a wide class of Artin representations $\varrho$ under the assumption that $r_{\an}(E, \varrho)\leq 2$. More precisely, $\varrho$ is allowed to be any irreducible constituent of the tensor product 
$$
\varrho = \varrho_1 \otimes \varrho_2
$$
of any pair of odd, two-dimensional Galois representations $\varrho_1$ and $\varrho_2$ of conductors $N_1$ and $N_2$ respectively, satisfying $$(N_E,N_1N_2)=1, \quad \det(\varrho_2) = \det(\varrho_1)^{-1} \quad \mbox{and} \quad r_{\an}(E, \varrho_1\otimes \varrho_2)=2.$$ After the ground-breaking works of Buzzard-Taylor, Khare-Wintenberger and others, we now know that these representations are modular and hence $\varrho_1$ and $\varrho_2$ are isomorphic to the Deligne-Serre Artin representation associated to eigenforms
$$
g\in M_1(N,\chi)\, \quad h\in M_1(N,\chi^{-1}),
$$
respectively, where the level $N$ might be taken to be the least common multiple of $N_E$, $N_1$ and $N_2$, and the nebentype $\chi = \det(\varrho_1)$ is the determinant of $\varrho_1$ regarded as a Dirichlet character.

Fix a prime $p\nmid N$ at which $g$ satisfies the classicality hypotheses C-C' introduced in \cite{DLR}.
The main conjecture of \cite{DLR} is a formula relating
\begin{itemize}

\item a $p$-adic iterated integral associated to the triple $(f,g,h)$ of eigenforms, to

\item an explicit linear combination of formal group logarithms of points of infinite order on $E(H)$, and

\item the $p$-adic logarithm of a Gross-Stark unit associated to the adjoint of $g$.

\end{itemize}

As shown in \cite[\S 2]{DLR}, the $p$-adic iterated integral can be recast as the value of the triple-product Harris-Tilouine $p$-adic $L$-function constructed in \cite[\S 4.2]{DR1} at a point of weights $(2,1,1)$. In this paper we place ourselves in a setting where $h$ is Eisenstein and Harris-Tilouine's $p$-adic $L$-function alluded to above can be replaced with the more standard Rankin $p$-adic $L$-function of Hida \cite[\S 7.4]{hida-book}. This observation is crucial for our purposes, for the latter $p$-adic $L$-function is more amenable to explicit computations.

The linear combination of logarithms of points mentioned above arises as the determinant of a $2 \times 2$ matrix introduced in \cite{DLR} that plays the role of a $p$-adic avatar of the regulator in the classical setting. Note however that this $p$-adic regulator does not coincide with the one considered by Mazur, Tate and Teitelbaum in \cite{MTT}, as the $p$-adic height function is replaced in \cite{DLR} with the formal logarithm on the elliptic curve. 

In some instances, as we shall see below, the logarithm of one of the points may be isolated in such a way that the conjectural expression suggested in \cite[Conjecture ES]{DLR} gives rise to a formula that allows to compute the point in terms of more accessible quantities which include the $p$-adic iterated integral and logarithms of global units and points that are rational over smaller number fields.

To describe more precisely our main result, 
let $K$ be an imaginary quadratic field of discriminant $-D_K$ with $D_K \geq 7$ and let $\cO_K$ denote its ring of integers. Let also $h_K = |\Cl_K|$ denote the class number, $g_K = [ \Cl_K : \Cl_K^2 ]$ be the number of genera, and $\chi_K$ be the quadratic Dirichlet character associated to $K/\Q$. 

Let $c\geq 1$ be a fixed positive integer relatively prime to $N_E$ and let $H/K$ denote the ring class field of $K$ of conductor $c$. Let 
$$\psi: \Gal(H/K) \,\lra \, \bar\Q^\times$$ be a character of finite order and
\begin{equation*}
g:=\theta(\psi) \in M_1(D_K c^2,\chi_K)
\end{equation*}
be the  theta series associated to $\psi$. This is an eigenform of level $D_Kc^2$ and nebentype character $\chi_K$. The form $g$ is cuspidal if and only if $\psi^2\ne 1$.

Let $\Q_{\psi}$ denote the finite extension generated by the values of $\psi$. 
Let $\varrho_\psi = \mathrm{Ind}_\Q^K(\psi)$ denote the odd two-dimensional Artin representation of $G_\Q$ induced by $\psi$ and write $V_{\psi}$ for the underlying two-dimensional $\Q_{\psi}$-vector space.

Fix now an odd rational prime $p=\wp \bar\wp\nmid N_E D_K c^2$ that splits in $K$. 
The choice of the ideal $\wp$ above $p$ determines a Frobenius element $\Fr_{p}$ in $\Gal(H/K)$. The eigenvalues of $\Fr_p$ acting on $V_{\psi} $ are
$$
\alpha = \psi(\Fr_p), \quad \beta= \bar\psi(\Fr_p) = \alpha^{-1}.
$$

As mentioned above, we assume that $g$ satisfies the classicality hypotheses C-C' of \cite{DLR} at the prime $p$. When $g$ is cuspidal, i.e.\,$\psi^2\ne 1$, one can verify\footnote{Indeed, when $g$ is cuspidal this hypothesis not only asks  that $\alpha \ne \beta$ but also that there should exist no real quadratic field $F$ in which $p$ splits such that $\varrho_g \simeq \Ind_F^\Q(\xi)$ for some character $\xi$ of $F$. However, in our CM setting, the existence of a character $\xi$ of a real quadratic field $F$ such that $\mathrm{Ind}_\Q^K(\psi)\simeq \mathrm{Ind}_\Q^F(\xi)$ implies that $\Gal(H/K)\simeq C_4$. Then $F$ is the single real quadratic field contained in the quadratic extension of $K$ cut out by $\psi^2$, and the condition $\alpha\ne \beta$ implies that $p$ can not split in $F$.} that in our setting the hypothesis is equivalent to assuming that $\alpha \ne \beta$. When $\theta(\psi)$ is Eisenstein, i.e.\,$\psi^2 = 1$, one expects\footnote{Indeed, when $\theta(\psi)$ is Eisenstein, it is shown in \cite[\S 1]{DLR} that a necessary condition for Hyptheses C-C' to hold is that $\alpha = \beta$, but this is automatically satisfied because $\psi^2=1$. As explained in loc.\,cit.,\,this is also expected to be a sufficient condition.} that hypotheses C-C' are always satisfied in our case.

\begin{remark}\label{assBD}
This assumption might appear unmotivated when one encounters it for the first time. In the case where $g$ is cuspidal, a striking recent result of Bella\"iche and Dimitrov \cite{BeDi} shows that hypotheses C-C' ensure that the points in the eigencurve associated to either of the two ordinary $p$-stabilisations $g_\alpha$, $g_\beta$ of $g$ is smooth and the weight map is \'etale at these points. 

Building on their work, it was further shown in \cite[\S 1]{DLR} that this assumtion implies that all overconvergent generalised eigenforms with the same system of Hecke eigenvalues as $g_\alpha$ are necessarily classical. This in turn guarantees that the $p$-adic iterated integral we will consider in \eqref{iterated} below is well-defined. 

Finally, it is also worth mentioning that hypotheses C-C' imply that the elliptic unit $u_{\psi^2}$ introduced in \eqref{ell-unit} below may be characterized, up to multiplication by scalars in $\Q_{\psi}^\times$, as follows: Up to $\Q_{\psi}^\times$ there is a single $\Gal(H/K)$-equivariant homomorphism $\varphi: \Q_{\psi}(\psi^2) \lra \cO_H^\times \otimes \Q_{\psi}$ and $u_{\psi^2}$ spans $\mathrm{Im}(\varphi)$.
\end{remark}

The choice of $\wp$ determines an embedding $K\hookrightarrow \Q_p$, that we extend to an embedding $H\hookrightarrow \C_p$. Write $H_p$ for the completion of $H$ with respect to this embedding.

Note that there is a natural isomorphism of $\Gal(H/K)$-modules $$E(H)^{\varrho_\psi} = E(H)^\psi \oplus E(H)^{\bar\psi}$$ where
$$
E(H)^{\psi} = \{ P\in E(H)\otimes \Q_{\psi}: \, P^\sigma = \psi(\sigma)P \, \mbox{ for all }\,\sigma\in \Gal(H/K)\},
$$
and $E(H)^{\bar\psi}$ is defined likewise. 


\vspace{0.2cm}

Throughout we impose the classical {\em Heegner hypothesis}:

\begin{ass}
There exists an integral ideal $\mathfrak N$ in $\cO_K$ such that $\cO_K/\mathfrak N \simeq \Z/N_E\Z$. 
\end{ass}

This assumption guarantees the existence of non-trivial {\em elliptic units} in $H^\times$ and {\em Heegner points} in $E(H)$ arising from the classical modular curve $X_0(N_E)$. In particular, there exists a canonical unit  (cf.\,e.g.\,\cite[\S 1]{DD},\cite[p.\,15-16]{GR73})
associated to the character $\psi^2$ defined as
\begin{equation}\label{ell-unit}
u_{\psi^2}=\begin{cases}
\mbox{any } p\mbox{-unit in } \cO_H[\frac{1}{p}]^\times \,\mbox{ satisfying } \, (u_\wp) = \wp ^{h_K}, & \mbox{ if } \,\psi^2 = 1, \\
\sum_{\sigma \in \Gal(H/K)} \psi^2(\sigma) u^\sigma \in \cO_H^\times\otimes \Q_{\psi}, & \mbox{ if } \,\psi^2\ne 1.
\end{cases}
\end{equation}

Likewise, 
associated to $\psi$ there is the canonical point (cf.\,e.g.\,\cite[{\bf I.} \S 6]{GZ})
\begin{equation}\label{Heegner-point}
P_\psi := \sum_{\sigma \in \Gal(H/K)} \psi(\sigma) x^\sigma \in E(H)^{\psi^{-1}} \subset E(H)^{V_\psi}.
\end{equation}

In \eqref{ell-unit} and \eqref{Heegner-point}, $u$ and $x$ are arbitrary choices of an elliptic unit and Heegner point with CM by $\cO_K$, respectively. Since, up to sign, all such choices lie in the same orbit under $\Gal(H/K)$, $u_{\psi^2}$ and $P_\psi$ do not depend on this, again up to sign. Notice that the convention is different to the one usually taken in literature in the sense that $\sigma$ acts on $P_\psi$ as $\psi(\sigma)^{-1}$ instead of $\psi(\sigma)$.

Set $N= \mathrm{lcm}(D_K c^2,N_E)$. Write $S_2(N)[f]$ (resp.\,$M_1(N,\chi_K)[g]$) for the subspace of $S_2(N)$ (resp.\,of $M_1(N,\chi_K)$) consisting of modular forms that are eigenvectors for all good Hecke operators $T_\ell$, $\ell\nmid N$, with the same eigenvalues as $f$ (resp. $g$). 

Fix a modular form $\testf = \sum_{n\geq 1} a_n(\testf)q^n\in S_2(N)[f]$ and let
$$
\testf^{[p]} = \sum_{p\nmid n} a_n(\testf)q^n
$$
denote the $p$-depletion of $\testf$.

Fix as well a modular form $\testg \in M_1(N,\chi_K)[g]$ and let
$$
\testg_\alpha(q) = \testg(q)-\beta_g \testg(q^p)\in M_1(pN,\chi_K)
$$
denote the ordinary stabilisation of $\testg$ on which $U_p$ acts with eigenvalue $\alpha$.

Associated to the pair $(\testf,\testg_\alpha)$ we may define a $p$-adic iterated integral as follows. Let $
h:=\mathrm{Eis}_1(1,\chi_{K})\in M_1(D_K,\chi_K)$ be the Eisenstein series  associated to the Dirichlet character $\chi_K$, as defined e.g.$\,$in \cite[\S 2.1.2]{BDR1}. Let also
\begin{equation}\label{h}
	\testh:= E_{1,\chi_{_{K,N}}}\in M_1(N,\chi_K)[h]
\end{equation}
denote the Eisenstein series of level $N$ in the isotypic eigenspace of $h$ that we introduce in \eqref{trueEisenstein} below. Let 
$$
e_{\text{ord}}: M^{\oc}_1(N,\chi_K) \lra M^{\oc, \ord}_1(N,\chi_K)
$$ 
denote Hida's ordinary idempotent on the space of overconvergent modular forms of weight $1$, tame level $N$ and tame character $\chi_K$. Let $$
e_{g_\alpha}^*: M^{\oc, \ord}_1(N,\chi_K) \lra M^{\oc,\ord}_1(N,\chi_K)[[g^*_\alpha]]
$$ 
denote the projection onto the generalised eigenspace attached to $g_\alpha^*$. 

Letting  $d=q\frac{d}{dq}$ denote Serre's $p$-adic derivative operator, the overconvergent ordinary modular form
$$
e_{g_\alpha}^* e_{\text{ord}}(d^{-1}\testf^{[p]}\times E_{1,\chi_{_{K,N}}})
$$
lies in the space of classical modular forms $M_1(pN,\chi_K)[g_\alpha^*]$ of weight $1$, level $pN$ and nebentype $\chi_K$, with coefficients in $\Q_p$, as explained in Remark \ref{assBD}.

Let $\testgamma_{g_\alpha}$ be an element in the $\Q_{\psi}$-dual space of $M_1(pN,\chi_K)[g^*_\alpha]$; more specifically, we take it to be the one associated to $\testg_\alpha$ in \cite[Proposition 2.6]{DLR}. By extending scalars, we may regard $\testgamma_{g_\alpha}$ as a $\Q_p(\psi)$-linear functional on $M_1(pN,\chi_K)[g^*_\alpha]$ as well.

Following \cite{DLR}, define
\begin{equation}\label{iterated}
\int_{\testgamma_{g_\alpha}} \testf \cdot E_{1,\chi_{_{K,N}}} := \testgamma_{g_\alpha} (e_{g_\alpha}^* e_{\text{ord}}(d^{-1}\testf^{[p]}\times E_{1,\chi_{_{K,N}}})) \in \C_p.
\end{equation}

In \cite[Theorem 3.3]{DLR} it was proved a statement that, specialised to our setting, asserts the following. Let $\log_p$ denotes the usual branch of the $p$-adic logarithm on $H_p^\times$ that satisfies $\log_p(p)=0$, and let $\log_{E,p}$ denote the formal group logarithm on $E/H_p$. 

\begin{theorem} \label{mainDLR}
Assume $r_{\an}(E/K,\psi)=1$. There exists a finite extension $L$ of $\Q_{\psi}$ and a scalar $\lambda(\testf,\testg) \in L$ such that
$$
\int_{\testgamma_{g_\alpha}} \testf \cdot E_{1,\chi_{_{K,N}}} = \lambda(\testf,\testg)\cdot \frac{\log^2_{E,p}(P_\psi)}{\log_p(u_{\psi^2})}.
$$
Moreover, there is a suitable choice of $\testf$ and $\testg$ such that $\lambda(\testf,\testg)\ne 0$.
\end{theorem}

Several questions arise naturally in the light of the above statement: 
\begin{itemize}
\item Can the field $L$ be determined?

\item Are there explicit choices of $\testf$ and $\testg$ for which $\lambda(\testf,\testg)\ne 0$?

\item Can the scalar $\lambda(\testf,\testg)$ be computed explicitly?

\item Does $\lambda(\testf,\testg)$ have any arithmetical meaning?
\end{itemize}

The aim of this note is answering these questions by proving an explicit formula for the scalar $\lambda(\testf,\testg)$ in terms of local and global arithmetic invariants of $E$ and $\psi$. In doing this, we prove as a particular case a formula that was already conjectured and verified numerically in \cite[Remark 3.4 and (45)]{DLR}.

While the main conjecture of \cite{DLR} may be regarded as a $p$-adic analogue of the {\em rank part} of the classical equivariant Birch and Swinnerton-Dyer conjecture, our main result provides a precise formula for the {\em leading term} in the particular setting we have placed ourselves. Hence Theorem \ref{main-theta} below may be regarded as a $p$-adic avatar of the formula for the leading term predicted by the conjecture of Birch and Swinnerton-Dyer, and  we hope it may suggest a $p$-adic variant of the classical {\em equivariant Tamagawa number conjecture}; more details on this may appear elsewhere.

The reader may also find Theorem \ref{main-theta} interesting from the computational point of view, as it provides an explicit $p$-adic formula for the Heegner point $P_\psi$ in $E(H_p)/E(H_p)_{\mathrm{tors}}$. Namely,
\begin{equation}
P_\psi = \exp_{E,p} \left( \sqrt{\frac{\log_p(u_{\psi^2})}{\lambda(\testf, \testg)}\cdot \int_{\testgamma_{g_\alpha}} \testf \cdot E_{1,\chi_{_{K,N}}}} \, \right).
\end{equation}

Let $\Q(f_N)$ denote the finite extension of $\Q$ generated by the roots of the Hecke polynomials $T^2-a_q(f)T+q$ for all primes $q\mid N$, $q\nmid N_E$. Note that if $\testf \in S_2(N)[f]$ is chosen to be a normalized eigenvector for all good {\em and bad} Hecke operators $T_\ell$ for all primes $\ell$, then the fourier coefficients of $\testf$ lie in $\Q(f_N)$. In a similar way, observe also that  if $\testg \in M_1(N)[g]$ is chosen to be an eigenvector for all good and bad Hecke operators, then the fourier coefficients of $\testg$ lie in $\Q_\psi$. Write $\Q_{\psi}(f_N)$ for the compositum of $\Q(f_N)$ and $\Q_{\psi}$.

\begin{theorem}\label{main-theta} 
\begin{enumerate}
\item[(i)] If $\testf$ and $\testg$ are chosen to be eigenvectors for all good and bad Hecke operators, then $L$ can be taken to be $\Q_\psi(f_N)$ and $\lambda(\testf,\testg)\ne 0$.

\item[(ii)] Assume that $D_K=N_E$ and $c=1$. Then the following formula holds true for $\testf=f$, $\testg=g$:
$$
\lambda(f,g) =  \frac{(p- a_p(f)\psi(\bar\wp)  + \psi^{2}(\bar\wp) )^2}{p} \cdot \frac{\lambda_0}{h_K g_K}  
$$
where
$$
\lambda_0 = \begin{cases} \frac{1}{p-1}   & \quad \mbox{if } \psi^2=1, \mbox{ that is to say, if } g \mbox{ is Eisenstein} \\
\frac{12}{p-(p+1)\psi^{-2}(\bar \wp)+\psi^{-4}(\bar \wp) } & \quad \mbox{if } \psi^2\ne 1, \mbox{ that is to say, if } g \mbox{ is cuspidal.}
\end{cases}	
$$
\end{enumerate}
\end{theorem}

Note that in the special case in which $N_E$ is prime and $\psi = 1$, we obtain
\begin{equation} \label{christmas}
\lambda(f,g) = \frac{|E(\F_p)|^2}{p(p-1)h_K}.
\end{equation}
As reported in \cite[Remark 3.4 and (45)]{DLR}, this formula was verified numerically in several examples, and here it is proved unconditionally.


The proof of Theorem  \ref{main-theta} actually provides an alternative proof of Theorem \ref{mainDLR} in the setting considered here. As in loc.\,cit.\,we compare the values of several $p$-adic $L$-functions at several points lying outside the region of interpolation, the main novelty with respect to \cite{DLR} being that we exploit Hida-Rankin $p$-adic $L$-function associated to the convolution of two Hida families, instead of the triple-product Harris-Tilouine $p$-adic $L$-function associated to a triple of Hida families. Since the former has been extensively studied in the literature, this alternative approach allows us to perform the explicit computations that are needed in order to derive the sought-after refined formula.

\medskip\medskip\noindent
{\small
 \thanks{{\bf Acknowledgements.}  This project has received funding from the European Research Council (ERC) under the European
Union's Horizon 2020 research and innovation programme (grant agreement No 682152).
}}

\section{Hecke characters, theta series and Katz's $p$-adic $L$-function}\label{Hecke}

Let $K/\Q$ be an imaginary quadratic field of discriminant $-D_K$ and let $\fc \subset \cO_K$ be an integral ideal. Let $I_\fc$ denote the group of fractional ideals of $K$ that are coprime to $\fc$.

A \emph{Hecke character} of infinity type $(\ka_1,\ka_2)$ of $K$ is a homomorphism 
$$
\psi: I_\fc \lra \C^\times$$ 
such that
\[
	\psi( (\alpha) ) = \alpha^{\ka_1} \overline{\alpha}^{\ka_2}
\]
for all $\alpha \equiv 1 \pmod{\fc}$. The \emph{conductor} of $\psi$ is the largest ideal $\fc_\psi$ for which this holds. Let us introduce some basic notations and terminology:

\begin{itemize}
	\item The norm map $\mathbf N_K:= |N^K_{\Q}|: I_\fc \lra \C^\times$ gives rise to a Hecke character of infinity type $(1,1)$ and conductor $1$.

	\item For any Hecke character $\psi$ of infinity type $(\ka_1,\ka_2)$ define $\psi'(\mathfrak a) = \psi(\overline{\mathfrak a})$, where $\overline{x}$ denotes complex conjugation; we say that $\psi$ is \emph{self-dual}, or \emph{anticyclotomic}, if $\psi  \psi'= \mathbf N_K^{\ka_1+\ka_2}$.

	\item A Hecke character of finite order (or infinity type $(0,0)$) can be regarded as a character of $G_K = \Gal(\bar K/K)$ via class field theory. We continue to denote by $\psi$ the resulting character, which is anticyclotomic.

\item The \emph{central character} $\varepsilon_\psi$ of $\psi$ is the single Dirichlet character satisfying
$$\psi|_{\Q} = \varepsilon_\psi \mathbf N^{\ka_1+\ka_2}_K.$$

\end{itemize}

The following lemma is well-known.

\begin{lemma} \label{thm-Cox} Let $\psi$ be a Hecke character of finite order. The following are equivalent:
\begin{enumerate}

\item $\psi$ is a ring class character.

\item $\Ind_\Q^K(\psi)$ is a self-dual representation.

\item The central character of $\psi$ is trivial.
\end{enumerate}
\end{lemma}

Given a Hecke character of $K$ of infinity type $(\ka_1,\ka_2)$, we can associate to it a theta series as follows. Define the quantities $$a_n(\psi) = \sum_{\mathfrak a \in I_{\fc_\psi}^n} \psi(\mathfrak a),$$
where $I_{\fc_\psi}^n$ is the set of the ideals in $I_{\fc_\psi}$ whose norm is $n$. Define also $a_0(1) = h_K/w_K$ and $a_0(\psi) = 0$ otherwise. As shown in \cite{Kani1},
\begin{equation} \label{theta-definition}
	\theta_\psi := \sum_{n\geq 0} a_n(\psi) q^{n} = \sum_{n\geq 0} a_n(\theta_\psi) q^n \in M_{\ka_1+\ka_2}(D_K N^K_\Q(\fc_\psi), \chi_K \varepsilon_\psi)
\end{equation}
is the $q$-expansion of a normalized newform of weight $\ka_1+ \ka_2$, level $D_K N^K_\Q(\fc_\psi)$ and nebentype $\chi_K \varepsilon_\psi$. Moreover $\theta_\psi$ is Eisenstein if and only if $\psi = \psi'$; otherwise $\theta_\psi$ is a cusp form. 

Associated to $\psi$ and $\theta_\psi$ there are the Hecke $L$-functions 
\begin{align*}
L(\psi,s ) := \prod_{\mathfrak p} \left( 1-\frac{\psi(\mathfrak p )}{N^K_\Q \mathfrak p ^s} \right)^ {-1} \qquad \text{and} \qquad 
L(\theta_\psi, s):=  \sum_{n\geq 1} \frac{a_n(\theta_\psi)}{n^s}.
\end{align*}

They can be extended to a meromorphic functions on $\C$. Since they coincide in a common region of convergence, they are actually the same function. From the definitions it is easy to verify that for any $k\in \Z$ we have
\begin{equation} \label{translation}
	L(\psi, s) = L(\psi \mathbf N_K^{k}, s+k).
\end{equation}

\subsection{Katz's two-variable $p$-adic $L$-function}\label{Katzs}

Assume $D_K\geq 7$ and let $\fc \subseteq \cO_K$ be an integral ideal. Fix a prime $p=\wp \bar\wp$ that splits in $K$. 

Denote by $\Sigma$ the set of Hecke characters of $K$ of conductor dividing $\fc$ and define $$\Sigma_K = \Sigma_K^{(1)}\cup \Sigma_K^{(2)}\subset \Sigma$$ to be the disjoint union of the sets
$$
\Sigma_K^{(1)}=\{ \psi \in \Sigma \mbox{ of infinity type } (\ka_1,\ka_2), \ka_1\leq 0, \ka_2\geq 1\},
$$
$$
\Sigma_K^{(2)}=\{ \psi \in \Sigma \mbox{ of infinity type } (\ka_1,\ka_2), \ka_1\geq 1, \ka_2\leq 0\}.
$$

For all $\psi \in \Sigma_K$, $s=0$ is a critical point for the Hecke $L$-function $L(\psi^{-1},s)$, and Katz's $p$-adic $L$-function is constructed by interpolating the (suitably normalized) values $L(\psi^{-1},0)$ as $\psi$ ranges over $\Sigma_K^{(2)}$. 

More precisely, let $\hat{\Sigma}_K$ denote the completion of $\Sigma_K^{(2)}$ with respect to the compact open topology on the space of functions on a certain subset of $\mathbf{A}^\times_K$, as described in \cite[\S 5.2]{BDP1}. By the work of Katz \cite{Katz1}, there exists a $p$-adic analytic function
$$
L_p(K): \hat{\Sigma}_K\lra \C_p
$$
which is uniquely characterized by the following interpolation property: for all $\psi \in \Sigma_K^{(2)}$ of infinity
type $(\ka_1,\ka_2)$,
\begin{equation}\label{int-Katz}
L_p(K)(\psi) = \mathfrak{e}_{K}(\psi) \mathfrak{f}_{K}(\psi) \frac{\Omega_p^{\ka_1-\ka_2}}{\Omega^{\ka_1-\ka_2}}L_\fc (\psi^{-1},0)
\end{equation}
where
\begin{itemize}
\item $L_\fc (\psi^{-1},s)$ is Hecke's $L$-function associated to $\psi^{-1}$ with the Euler factors at primes dividing $\fc$ removed,

\item $\Omega_p\in \C_p^\times$ is a $p$-adic period attached to $K$, as defined in \cite[(140)]{BDP1}, \cite[(25)]{BDP2},

\item $\Omega\in \C^\times$ is the complex period associated to $K$ as defined in \cite[(137)]{BDP1},

\item $\mathfrak{e}_{K}(\psi) = (1-\frac{\psi(\wp)}{p}) (1-\psi^{-1}(\bar\wp))$, and  $\mathfrak{f}_{K}(\psi) = \frac{(\ka_1-1)!\cdot D_K^{\ka_2/2}}{(2 \pi)^{\ka_2} }$.

\end{itemize}





The following result is commonly known as Katz's Kronecker $p$-adic limit formula. It computes the value of $L_p(K)$ at a finite order character $\psi$ of $G_K$, which lies outside the region of interpolation (cf.\,\cite[\S 10.4, 10.5]{Katz1}, \cite[p.\,90]{Gr}, \cite[Ch.\,II, \S 5.2]{dS}): 
\begin{equation}\label{KatzKronecker}
L_p(K)(\psi) = \mathfrak f_p(\psi) \cdot \log_p(u_{\psi^{-1}}),
\end{equation}
where
\begin{equation}\label{Katz-GZ}
\mathfrak f_p(\psi) 
= \begin{cases}\frac{1}{2} (\frac{1}{p}-1)  & \mbox{ if } \psi=1 \\
 \frac{-1}{24c}  (1-\psi(\bar\wp)) (1-\frac{\psi(\bar\wp)}{p}) &  \mbox{ if } \psi\ne 1.
\end{cases}
\end{equation}
Here $c>0$ is the smallest positive integer in the conductor ideal of $\psi$.

\section{Classical and $p$-adic Rankin $L$-fuctions}


\subsection{Eisenstein series}
Let $\chi : (\Z/N_\chi\Z)^\times \to \C$ be a Dirichlet character of conductor $N_\chi$ and let $\Q_\chi$ denote the finite extension of $\Q$ generated by the values of $\chi$. For any multiple $N$ of $N_\chi$ let $\chi_N$  denote the character mod $N$ induced by $\chi$. 

For every positive integer $k\geq 1$, let $M_k(N,\chi_N)$ and $S_k(N,\chi_N)$ denote the spaces of holomorphic (resp.\,cuspidal) modular forms of weight $k$, level $N$ and character $\chi_N$. 

We also let $M^{\an}_k(N,\chi_N)$ and $S^{\an}_k(N,\chi_N)$ denote the space of real-analytic functions on the upper-half plane with the same transformation properties
under $\Gamma_0(N)$ and having bounded growth (resp.\,rapid decay) at the cusps. On these spaces one may define the Shimura-Maass derivative operator
$$ \delta_k := \frac{1}{2\pi i}
\left( \frac{d}{dz}  + \frac{i k}{2y}\right): M^{\an}_k(N,\chi_N) \lra M^{\an}_{k+2}(N,\chi_N).$$

For every $k\geq 1$ such that $\chi(-1)=(-1)^k$, define the {\em non-holomorphic Eisenstein series} of weight $k$ and level $N$ attached to the character $\chi_N$ as the function on $\mathcal H \times \C$ given by the rule
\begin{equation}
	\tilde{E}_{k, \chiNN}(z,s) = \sum_{(m,n)\in \Z^2\setminus\{(0,0)\}} \frac{\chi_N^{-1}(n)}{(mNz+n)^k} \cdot \frac{y^s}{|mNz+n|^{2s}}.
\end{equation}

Although a priori this series only converges for $\Re(s) > 1 -k/2$, it can be extended to a meromorphic function in the variable $s$ on the whole complex plane $\C$. For $k>2$, or $k \geq 1$ but $\chi \ne 1$,  the series arising by setting $s=0$ is actually holomorphic in $z$ and gives rise to a modular form
\[
	\tilde{E}_{k, \chiNN}(z) := \tilde{E}_{k, \chiNN}(z,0) \in M_k(N, \chi).
\]

For any value of $s$, the series $\tilde{E}_{k, \chiNN}(z,s)$ belongs to $M_k^{\an}(N, \chi)$ and
one verifies that 
\[
	\delta_k \tilde{E}_{k,\chiNN}(z,s) = -\frac{s+k}{4\pi} \tilde{E}_{k+2,\chiNN}(z,s-1).
\]
Moreover, if we let $\delta_k^t = \delta_{k+2t-2} \cdots \delta_{k+2} \delta_k$ denote the $t$-fold iterate of the Shimura-Maass operator, then for all $t\leq (k-1)/2$ we have
\begin{equation}\label{ShimuraMaass}
	\tilde{E}_{k,\chiNN}(z, -t) = \frac{(k-2t-1)!}{(k-t-1)!}(-4\pi)^t \delta_{k-2t}^t \tilde{E}_{k-2t,\chiNN}(z).
\end{equation}

Define a normalization $E_{k,\chiNN}\in M_k(N, \chi)$ of the Eisenstein series as
\begin{equation} \label{trueEisenstein}
	E_{k,\chiNN}(z) = \frac{N^k (k-1)! }{2 (-2\pi i)^k \tau(\chi^{-1})}\cdot \tilde{E}_{k,\chiNN}(z)
\end{equation}
where 
$$
\tau(\chi) = \sum_{a=1}^{N_\chi} \chi(a) e^{\frac{2\pi a i}{N_\chi}}
$$
is the Gauss sum associated to the Dirichlet character $\chi$.

Let  $\sigma_{k-1,\chi}$ denote the function on the positive integers defined as $\sigma_{k-1,\chi}(n) := \sum_{d\mid n} \chi(d)d^{k-1}$. Then $E_{k,\chi}$ is a newform of level $N_\chi$ and its $q$-expansion is
\begin{equation}
	E_{k,\chi}(q) = \frac{L(\chi, 1-k)}{2}+\sum_{n=1}^\infty \sigma_{k-1, \chi}(n) q^n \in M_k(N_\chi,\chi), \qquad q=e^{2\pi i z}.
\end{equation}

When $N>N_\chi$, $E_{k,\chi_N}(q)$ is a $\Q_\chi$-linear combination of the modular forms $E_{k,\chi}(q^d)$ as $d$ ranges over the positive divisors of $N/N_\chi$ (cf.\,\cite[(3.3), (3.4)]{Sh76} for the precise expression). In particular $E_{k,\chi_N}$ is an eigenform with respect to all good Hecke operators $T_{\ell}$, $\ell\nmid N$ with the same eigenvalues of $E_{k,\chi}$.

\subsection{Classical Rankin's $L$-function}

Recall that the Petersson scalar product on the space of real-analytic modular forms $S_l^{\an}(N,\chi) \times M_l^{\an}(N,\chi)$ is given by:
\begin{equation} \label{Peterssondef}
	\langle f_1, f_2 \rangle_{l,N} := \int_{\Gamma_0(N) \backslash \mathcal H} y^l \overline{f_1(z)} f_2(z) \frac{dx dy}{y^2}.
\end{equation}
Let $$g_l = \sum_{n\geq 1} a_n(g_l)q^n \in S_l(N, \chi_g), \quad f_k=\sum_{n\geq 1} a_n(f_k)q^n \in M_k(N, \chi_f)$$ be two eigenforms of weights $l > k \geq 1$ and nebentype characters $\chi_g$ and $\chi_f$ respectively. We do not assume $g_l$ and $f_k$ to be newforms, but we do assume them to be eigenvectors for all good and bad Hecke operators.

Set $\chi := (\chi_g \chi_f)^{-1}$ and let $g_l^*  = \sum_{n\geq 1} \bar{a}_n(g_l)q^n \in S_l(N, \chi_g^{-1})$ denote the modular form whose fourier coefficients are the complex conjugates of those of $g_l$. 

For a rational prime $q$ we let $(\alpha_q(g_l),\beta_q(g_l))$ denote the pair of roots of the Hecke polynomial $X^2-a_q(g_l)X+\chi_{g,N}(q)q^{l-1}$, that we label in such a way that $\ord_q(\alpha_q(g_l)) \leq \ord_q(\beta_q(g_l))$. Note that $(\alpha_q(g_l),\beta_q(g_l)) = (a_q(g_l),0)$ when $q\mid N$. If the weight is $l=1$ and $q\nmid N$ then both $\alpha_q(g_l)$ and $\beta_q(g_l)$ are $q$-units; in that case we just choose an arbitrary ordering of this pair. Adopt similar notations for $f_k$. 

Define the {\em Rankin $L$-function} of the convolution of $g_l$ and $f_k$ as the Euler product 
\begin{equation}\label{EulHR}
L(g_l\otimes f_k,s) = \prod_q L^{(q)}(g_l\otimes f_k,s),
\end{equation}
where $q$ ranges over all prime numbers and
\begin{align*}
L^{(q)}(g_l\otimes f_k,s) =	 &(1-\alpha_q(g_l)\alpha_q(f_k) q^{-s})^{-1} (1-\alpha_q(g_l)\beta_q(f_k) q^{-s})^{-1} \\
						&\times (1-\beta_q(g_l)\alpha_q(f_k) q^{-s})^{-1} (1-\beta_q(g_l)\beta_q(f_k) q^{-s})^{-1}.
\end{align*}
\begin{proposition}[Shimura] For all $s\in \C$ with $\Re(s) >>0$ we have:
\begin{equation} \label{Shimura}
	L(g_l\otimes f_k,s) = \frac{1}{2} \frac{(4\pi)^s}{\Gamma(s)} \langle g_l^*(z), \tilde{E}_{l-k, \chiNN}(z,s-l+1) \cdot f_k(z) \rangle_{l,N}
\end{equation}
\end{proposition}

Choose integers $m,t$ such that $$l = k+m+2t \quad \mbox{ and set } \quad j = (l+k+m-2)/2 = l-t-1.$$ For $m\geq 1$ and $t\geq 0$, evaluating equation \eqref{Shimura} at $s=j$ and using equations \eqref{ShimuraMaass} and \eqref{trueEisenstein}  one finds that
\begin{equation} \label{rankinmethod}
	\fhr(l,k,m) \cdot L(g_l\otimes f_k, j) = \langle g_l^*(z), \delta_m^t E_{m,\chiNN}(z) \cdot f_k(z) \rangle_{l,N},
\end{equation}
where
\begin{equation}
	\fhr(l,k,m) = \frac{(-1)^t (m+t-1)! (j-1)! (iN)^m }{2^{l-1}(2\pi)^{l+m-1}\cdot \tau(\chi^{-1})}.
\end{equation}

\subsection{Critical values, algebraicity and the Hida-Rankin $p$-adic $L$-function}\label{secHida}

Since we are assuming $l>k\geq 1$, an integer $j$ is critical for $L(g_l\otimes f_k,s)$ if and only if $j\in [k, l-1]$. We shall restrict our attention to critical integers in the range
\[
	j \in \Bigl[\frac{l+k-1}{2}, l-1 \Bigr],
\]
and for a given such $j$ we set $$t:= l-j-1 \quad \mbox{and} \quad m:= l-k-2t.$$ From equation \eqref{rankinmethod} it follows that
\begin{equation} \label{rankinmethod2}
	\fhr(l,k,m) \cdot L(g_l\otimes f_k,j) = \langle g_l^*(z), \delta_m^t E_{m,\chiNN}(z) \cdot f_k(z) \rangle_{l,N}.
\end{equation}
Define the algebraic part of $L(g_l\otimes f_k, j)$ as in \cite[(9)]{BDR1}:
\begin{equation} \label{algebraicpart}
	L^{\alg}(g_l\otimes f_k,j) := \fhr(l,k,j)\frac{L(g_l\otimes f_k, j)}{\langle g_l^*, g_l^* \rangle_{l,N}} =\frac{\langle g_l^*(z), \delta_m^t E_{m,\chiNN}(z) \cdot f_k(z) \rangle_{l,N}}{\langle g_l^*, g_l^* \rangle_{l,N}}.
\end{equation}

Fix a prime $p\nmid N$ at which $g_l$ is ordinary and let $g_{l,\alpha} \in S_l(Np, \chi_g)$ denote the ordinary $p$-stabilisation of $g_l$ on which $U_p$ acts with eigenvalue $\alpha_p(g_l)$.

Let $\hg$ be a Hida family of ordinary overconvergent modular forms of tame (but not necessarily primitive) level $N$, passing through $g_{l,\alpha}$. The Hida family is parametrized by a finite \'etale rigid-analytic cover $U_\hg$ of weight space $\cW$. By shrinking $U_\hg$ if necessary, we assume that $U_\hg(\Z_p)$ is fibered over a single residue class modulo $p-1$ of $\cW(\Z_p) = \Z_p^\times \simeq \Z/(p-1)\Z \times \Z_p$. By a slight abuse of notation which shall be harmless for our purposes, we identify throughout points in $U_\hg$ with their image in $\cW$ under the weight map.


With these conventions, for every classical weight $l \in U_\hg\cap \Z^{\geq 2}$ we let
\begin{equation}\label{gl}
	g_l\in S_l(N,\chi_g)
\end{equation}
denote the classical cusp form whose ordinary $p$-stabilisation is the specialisation of $\hg$ at an arithmetic point in $U_{\hg}$ of weight $l$.


Define $$\ehr(l,k,j) := \frac{\cE(g_l,f_k,j)}{\cE_1(g_l)\cE_0(g_l)}$$ where:
\begin{eqnarray*}
	\cE(g_l,f_k,j)&=&
	(1-\beta_p(g_l)\alpha_p(f_k) p^{t-l+1})(1-\beta_p(g_l)\beta_p(f_k) p^{t-l+1}) \\
	&\times& (1-\beta_p(g_l)\alpha_p(f_k) \chi(p) p^{t-l+1})(1-\beta_p(g_l)\beta_p(f_k) \chi(p) p^{t-l+1}),\\
	\cE_1(g_l)&=&1-\beta_p(g_l)^2 p^{-l}, \\
	\cE_0(g_l)&=&1-\beta_p(g_l)^2 p^{1-l}.
\end{eqnarray*}

In \cite[\S 7.4]{hida-book} Hida constructed a three-variable $p$-adic $L$-function interpolating central critical values of the Rankin $L$-function associated to the convolution of two Hida families of modular forms. For the purposes of this note it will suffice to retain the restriction of this $p$-adic $L$-function to the one-dimensional domain afforded by $U_\hg$. Here we will work with the notations and normalisations adopted in \cite{BDR1}. In order to introduce this $p$-adic $L$-function properly, we shall make use of the following operators on the space of overconvergent $p$-adic modular forms, that we introduce here by describing their action on $q$-expansions:

\begin{itemize}

\item Serre's derivative operator $d =q\cdot \frac{d}{dq}$, which may be regarded as the $p$-adic avatar of the Shimura-Maass operator invoked above.

\item The $U$ and $V$ operators acting on a modular form $\phi=\sum a_n q^n$ by the rules $$U(\phi) = \sum a_{pn}q^n \quad \mbox{and} \quad V(\phi) = \sum a_n q^{pn}.$$

\item Hida's ordinary idempotent  $e_{\ord} := \lim U_p^{n!}$.

\item $p$-depleting operator: $\phi^{[p]} := (1- UV)(\phi) = \sum_{p\nmid n} a_n q^n$. 

\end{itemize}

Let $\cE$ denote the Kuga-Sato variety fibered over $X_1(N)$ and let $\cE^{l-2}$ denote the fiber product of $l-2$ copies of $\cE$ over $X_1(N)$. The dimension of $\cE^{l-2}$ is $l-1$ and the middle de Rham cohomology group $H^{l-1}_{\dR}(\cE^{l-2}/\C_p)$ contains the canonical regular differential form $\omega_{g^*_l}$ associated to $g^*_l$. Let  
$$
\langle\ ,\ \rangle: H^{l-1}_{\dR}(\cE^{l-2}/\C_p) \times H^{l-1}_{\dR}(\cE^{l-2}/\C_p) \, \lra \, \C_p
$$ 
denote the non-degenerate Poincar\'e pairing on $H^{l-1}_{\dR}(\cE^{l-2}/\C_p)$.

The $g^*_l$-isotypic component of
$H^{l-1}_{\dR}(\cE^{l-2}/\C_p)$ is two-dimensional over $\C_p$ and it admits a one-dimensional 
{\em unit root subspace}, denoted $H^{l-1}_{\dR}(\cE^{l-2}/\C_p)^{\ur}$,
on which the Frobenius endomorphism acts as multiplication by a $p$-adic unit.
This unit root subspace is complementary to the line spanned by $\omega_{g^*_l}$. There is thus a single class
$\eta_{g^*_l}^{\ur} \in H^{l-1}_{\dR}(\cE^{l-2}/\C_p)^{\ur}$ satisfying
$$ \langle \omega_{g^*_l},\eta_{g^*_l}^{\ur} \rangle = 1.$$


Let $\Lambda_g$ denote the algebra of Iwasawa functions on $U_{\hg}$ and let $\mathcal K_{\hg}$ denote the fraction field of $\Lambda_{\hg}$.  As shown in \cite{BDR1}, there exists a unique $p$-adic $L$-function $L_p(\hg,f) \in \mathcal K_{\hg}$ satisfying
\begin{equation} \label{definhida}
L_p(\hg,f)(l) = \frac{1}{\cE_0(g_l,f_k,j)}\langle \eta^{ur}_{g^*_l}, e_{\ord}(d^t E_{m,\chiNN}^{[p]} \cdot f_k) \rangle
\end{equation}
for all $l\in U_{\hg} \cap \Z_{\geq 2}$. Note that $d^t E_{m,\chiNN}^{[p]} \cdot f_k$ is an overconvergent modular form of weight $l\geq 2$ and hence its ordinary projection is classical by a celebrated theorem of Hida and Coleman. This way $e_{\ord}(d^t E_{m,\chiNN}^{[p]} \cdot f_k)$ gives rise to a regular differential form  in $H^{l-1}_{\dR}(\cE^{l-2}/\C_p)$, which may therefore be written as 
$$
e_{\ord}(d^t E_{m,\chiNN}^{[p]} \cdot f_k) = C(g_l,f_k)\cdot \omega_{g_l^*} + (\mbox{Other terms with respect to an orthogonal basis})
$$

In plane terms, the above formula might be read as
$$
L_p(\hg,f)(l) = \frac{C(g_l,f_k)}{\cE_0(g_l,f_k,j)}.
$$

The $p$-adic $L$-function $L_p(\hg,f_k)$ deserves its name because it obeys and it is characterized by the following interpolation formula that relates the values of $L_p(\hg,f_k)$ at integers $l>k$ to critical values of a Rankin $L$-function. In fact, it follows from \eqref{algebraicpart} and \eqref{definhida} that for every $l\in U_{\hg} \cap \Z_{l>k}$ we have
\begin{equation} \label{interpolhida}
	L_p(\hg,f_k)(l) = \ehr(l,k,j) \cdot L^{\alg}(g_l\otimes f_k,j) = \ehr(l,k,j) \fhr(l,k,j)\frac{L(g_l\otimes f_k, j)}{\langle g_l^*, g_l^* \rangle_{l,N}}. 
\end{equation}
We call this function the {\em Hida-Rankin $p$-adic $L$-function} associated to $\hg$ and $f_k$.

Note  that $l=1$ lies outside the above region of interpolation. Assume that there exists an eigenform $g_1\in M_1(N_g,\chi_g)$  such that the ordinary $p$-stabilisation $g_{1,\alpha}$ arises as the specialisation of $\hg$ at an arithmetic point of weight $1$ in $U_{\hg}$. Notice that this is not always the case, as a Hida family may in general specialize to non-classical overconvergent modular forms at points of weight one.

The following result provides a formula for the value of $L_p(\hg,f)$ at $l=1$ and was proved in \cite[\S 2]{DLR}. 
Assume that $k=2$ and $\chi_f=1$, so that $\chi=\chi_g^{-1}$, and set $f=f_2$. Recall the $p$-adic iterated integrals introduced in \eqref{iterated}.

\begin{proposition} \label{vitasalva}
 Let $h:=E_{1, \chiNN} \in M_1(N,\chi_N)$. Then $L_p(\hg,f)$ has no pole at $l=1$ and
\begin{equation} 
	L_p(\hg,f)(1) = \int_{\gamma_{g_1}} f \cdot h.
\end{equation}
\end{proposition}

\begin{proof}
Combine \cite[remark 4.5]{DR1}, \cite[proposition 4.6]{DR1} and \cite[proposition 2.6]{DLR}.
\end{proof}




\subsection{Bertolini-Darmon-Prasanna's $p$-adic $L$-function}\label{BDPs}

As in the introduction, let $E/\Q$ be an elliptic curve of conductor $N_E$ and let $f\in S_2(N_E)$ denote the eigenform associated to it by modularity. As in \S \ref{Katzs}, let also $K$ be an imaginary quadratic field of discriminant $-D_K \leq -7$ fulfilling the Heegner hypothesis. 

Let $\fc \subseteq \cO_K$ be an integral ideal and set $N=\mathrm{lcm}(N_E,D_K N_{K/\Q}(\fc))$. 
Let $\Sigma$ denote the set of Hecke characters of $K$ of conductor dividing $\fc$. For any Hecke character $\psi\in \Sigma$ of infinity type $(\ka_1,\ka_2)$, let $L(f,\psi,s)$ denote the $L$-function associated to the compatible system of Galois representations afforded by the tensor product  $\varrho_{f|G_K} \otimes \psi$ of the (restriction to $G_K$ of) the Galois representations attached to $f$ and the character $\psi$.


As usual, $L(f,\psi,s) = \prod_q L^{(q)}(q^{-s})$ is defined as a product of Euler factors ranging over the set of prime numbers. The Euler factors at the primes $q$ such that $q\nmid N$ are exactly the same as that of the Rankin L-series $L(\theta_\psi \otimes f,s)$ introduced above, but may differ at the primes $q$ such that $q\mid N$ (details can be found in \cite{Gro} for $f$ modular form of weight 2).

Let $\Sigma_{f,K}\subset \Sigma$ be the subset of Hecke characters of trivial central character in $\Sigma$ for which $L(f, \psi^{-1},s)$ is self-dual and $s=0$ is its central critical point. This set is naturally the disjoint union of the three subsets 
$$
\Sigma^{(1)}_{f,K} = \{ \psi \in \Sigma_{f,K} \,\mbox{ of infinity type }\,(1,1)\},
$$
$$
\Sigma^{(2)}_{f,K} = \{ \psi \in \Sigma_{f,K} \,\mbox{ of infinity type }\,(2+\ka,-\ka), \ka\geq 0\}
$$
and
$$
\Sigma^{(2')}_{f,K} = \{ \psi \in \Sigma_{f,K} \,\mbox{ of infinity type }\,(-\ka,\ka+2), \ka\geq 0\}.
$$

Fix a prime $p=\wp \bar\wp$ that splits in $K$.  
Each of the three sets $\Sigma^{(1)}_{f,K}$, $\Sigma^{(2)}_{f,K}$ and $\Sigma^{(2')}_{f,K}$ are dense in the completion $\hat{\Sigma}_{f,K}$ of  $\Sigma_{f,K}$ with respect to the $p$-adic compact open topology as explained in \cite[\S 5.2]{BDP1}. As shown in \cite{BDP1}, there exists a unique $p$-adic analytic function
\begin{equation*}
L_p(f,K): \hat{\Sigma}_{f,K} \lra \C_p
\end{equation*}
interpolating the critical values $L(f,\psi^{-1},0)$ for $\psi \in \Sigma^{(2)}_{f,K}$, suitably normalized.

We refer to $L_p(f,K)$ as the Bertolini-Darmon-Prasanna $p$-adic Rankin $L$-function attached to the pair $(f,K)$. 


Consider a character $\Psi \in \Sigma^{(2)}_{f,K}$ of type $(2+\ka, -\ka)$, for $\ka \geq 0$. According to \cite[\S 5.2]{BDP1} the interpolation formula for $L_p(f,K)$ reads
\begin{equation}\label{int-BDP}
L_p(f,K)(\Psi) =\cE_c \cdot  \mathfrak{e}_{\mathrm{BDP}}(\Psi)   \cdot \mathfrak{f}_{\mathrm{BDP}}(\Psi)\cdot \frac{ \Omega_p^{4\ka+4}  }{\Omega^{4\ka+4}} \cdot L(f,\Psi^{-1},0),
\end{equation}
where

\begin{itemize}

\item
$\cE_c =  \prod_{q\mid c} \frac{q-\chi_K(q)}{q-1}$, 	$\qquad \mathfrak{e}_{\mathrm{BDP}}(\Psi) = (1-a_p(f) \Psi^{-1}(\bar\wp) + p\Psi^{-1}(\bar\wp) ^2 )^2$,
\item
	$\mathfrak{f}_{\mathrm{BDP}}(\Psi)  = \left(\frac{2\pi}{\sqrt{D_K}}\right)^{2\ka+1}  \ka!(\ka+1)!  \cdot 2^{\sharp q\mid (D_K,N_E)} \cdot \omega(f,\Psi)^{-1}$ 
\end{itemize}
with $\omega(f, \Psi)$ as defined in \cite[(5.1.11)]{BDP1}. 

If $\psi$ is a finite order anticyclotomic character of conductor $c\mid \mathfrak c$, then $\psi \mathbf{N}_K$ lies outside the region of interpolation and the main theorem of \cite{BDP2} asserts that
\begin{equation}\label{thm:bdp-gz}
L_p(f,K)(\psi^{-1} \mathbf{N}_K) = \mathfrak f_p(f,\psi) \times \log_{\omega_E}(P_{\psi})^2
\end{equation}
where $\mathfrak f_p(f,\psi) = (1-\psi(\bar\wp) p^{-1} a_p(f) + \psi^2(\bar\wp) p^{-1})^2$.

\section{Proof of the main theorem}

Recall the three eigenforms that have been fixed at the outset:
\[
f\in S_2(N_E), \quad	 g=\theta_\psi \in M_1(D_Kc^2, \chi_K)_{\Q_{\psi}}, \quad h=\mathrm{E}_{1,\chi_K} \in M_1(D_K,\chi_K).
\]
Choose and fix modular forms $\testf \in S_2(N)[f]$ and $\testg \in M_1(N,\chi_K)[g]$ that are eigenforms for all good and bad Hecke operators. For the sake of concreteness, we may write
\[
\testf(z) = \sum_{d \mid \frac{N}{N_E}} \mu_d(f) f(d z), \quad	\testg(z) = \sum_{d \mid \frac{N}{Dc^2}} \mu_d(g) g(d z)
\]
where $\mu_d(f)$, $\mu_d(g)$ belong to the number field $\Q_{\psi}(f_N)$ introduced in the paragraph preceding Theorem \ref{main-theta}.




Fix a prime $p\nmid N$ that splits in $K$ and choose a root $\alpha$ of $T^2-a_p(g)T+1$. As in the introduction let $\testg_\alpha\in M_1(Np,\chi_K)$ denote the $p$-stabilisation of $\testg$ on which $U_p$ acts with eigenvalue $\alpha$.

There exists a unique $p$-adic Hida family $\hg$ of theta series of tame level $D_Kc^2$ and tame character $\chi_K$ passing through $g_\alpha$. As in \S \ref{secHida} and \eqref{gl}, for every classical weight $l \in U_\hg\cap \Z^{\geq 2}$ we let
\begin{equation}
	g_l\in S_l(D_Kc^2,\chi_K)
\end{equation}
denote the classical newform whose ordinary $p$-stabilisation is the specialisation of $\hg$ at an arithmetic point in $U_{\hg}$ of weight $l$. Notice that at $l=1$ the modular form $g_l$ is still classical, by assumption, but it might not be a cusp form. In this case we have $g_1 = g\in M_1(D_kc^2, \chi_K)$.

For every such $l$ we can also explicitly describe the Hecke character $\psi_{l-1}$ of conductor $c$ and infinity type $(0,l-1)$ such that $g_l=\theta_{\psi_{\ell-1}}$. We do the same construction done in \cite[p. 235-236]{hida-book} and \cite[\S 3]{DLR}, but we slightly change conventions. Pick a $p$-adic unitary character $\lambda$ of conductor $c\bar\wp$ and infinity type $(0,1)$. Define then $\psi_{\ell-1}(\mathfrak q) := \psi(\mathfrak q) \langle \lambda(\mathfrak q) \rangle ^{\ell-1}$, then define $\psi_{\ell-1}(\bar\wp) := p^{\ell-1}/\psi_{\ell-1}(\wp)$. At any prime $q=\mathfrak q \bar{\mathfrak q}$ which is splits in $K$ we have
\begin{equation} \label{useful}
	\alpha_q(g_l) = \psi_{l-1}(\mathfrak q), \qquad \beta_q(g_l) = \psi_{l-1}(\bar{\mathfrak q}).
\end{equation}
Our running hypothesis on $p$ and the Heegner assumption imply that this is the case for $q=p$ and for any of the primes dividing $N$ but not $D_K$. 

Together with $\hg$, it will also be useful to consider  the $\Lambda$-adic family of modular forms
 $$\htg(q) = \sum_{d \mid \frac{N}{Dc^2}} \mu_d(g) \hg(q^d)$$ 
arising from our choice of $\testg$. Note that $\htg$ specializes to $\testg_\alpha$ in weight one.


Let $U_{\hg}^\circ$ denote the subset of $U_{\hg}$ consisting of classical points of weights of the form $2l+3\equiv 1 \pmod{p-1}$ with $l\in \Z_{\geq 1}$. According to the conventions about Hida families adopted in \S \ref{secHida}, note that $U_{\hg}^\circ$ is dense in $U_{\hg}$.

Set $j = l+2$, $t=l$ and $m=1$. Then the interpolation formula of \eqref{interpolhida} at points in $U_{\hg}^\circ$ reads as follows:
\begin{equation} \label{1variable}
	L_p(\htg,\testf)(2l+3,l+2) = \ehr(l) \cdot \fhr(l) \cdot \frac{L(\testg_{2l+3}\otimes \testf, l+2)}{\langle \testg_{2l+2}^*, \testg_{2l+2}^* \rangle_{l,N}} 
\end{equation}
where $\ehr(l) = \cE(2l+3,2,l+2)/(\cE_1(2l+3) \cE_0(2l+3))$, with:
 \begin{eqnarray*}
 \cE(2l+3,2,l+2)&=&
 (1-\alpha_f \beta_{\theta_{2l+3}} p^{-(l+2)})^2
 (1- \beta_f \beta_{\theta_{2l+3}} p^{-(l+2)})^2,\\
 \cE_1(2l+3)&=&1-\beta_{\theta_{2l+3}}^2 p^{-2l-3}, \\
\cE_0(2l+3)&=&1-\beta_{\theta_{2l+3}}^2 p^{-2l-2},
 \end{eqnarray*}
and
\begin{equation}\label{foo-BDR}
\mathfrak{f}_{\mathrm{HR}}(l)  = \frac{(-1)^l l!(l+1)!\cdot i \cdot N}{ 2^{4l+5} \pi^{2l+3}\cdot \tau(\chi_K)} = (-1)^l  \cdot \frac{l!(l+1)!\cdot N}{ 2^{4l+5} \pi^{2l+3}\cdot \sqrt{D_K}}
\end{equation}  
Here the last equality holds because $\tau(\chi_K) = i \sqrt{D_K}$ (cfr. \cite{Roh}).


We now need the following two basic lemmas.

\begin{lemma}\label{proposition1} There exists a meromorphic function $ \cEul_N(s)$ such that the following factorisation formula holds
\[
	L(\testg \otimes \testf, s) = \cEul_N(s) \cdot L(f, \psi, s)
\]
and $\cEul_N(1) \in \Q_\psi(f_N)^\times$.
\end{lemma}

\begin{proof}
	Since we choose $\testf$ and $\testg$ to be eigenforms for all Hecke operators, we can use Euler products to compare the two L-functions. Their Euler factors are equal outside primes $q\mid N$. Then the factor $\cEul_N(s)$ is the product of the bad Euler factors at $q\mid N$ which encode this discrepancy. If we use the definitions of equation \eqref{EulHR} and of \cite[equation (20.2)]{Gro} we find:
	\begin{equation} \label{grossplus}
		\cEul_N(s) = \frac{\prod_{q\mid (N_E,D_K)} (1+q^{-s}) \prod_{q||N_E, q\nmid D_K} (1-a_q(f)q^{-s})^2 \prod_{q\mid D_K, q \nmid N_E} (1-a_q(f)a_q(g)q^{-s}+q^{1-2s} )}{\prod_{q\mid N} (1-\alpha_q(\testf)\alpha_q(\testg)q^{-s})}
	\end{equation}
	so that $\cEul_N(1)$ is a finite product of non-zero terms which lies in $\Q_\psi(f_N)$.
\end{proof}

\begin{lemma} \label{prop1}
For $l=-1$ and $l\in U_{\hg}^\circ$, let $\Psi_l$ be the Hecke character  $\Psi_l = (\psi_{2l+2})^{-1} \mathbf{N}^{l+2}$ of conductor $c$ and infinity type $(l+2,-l)$.
Then there exists a number
$
 \cEul^{HR}_N(l) \in \Q_{\psi}(f_N)
$
such that the following equality of critical $L$-values hold:
\[
			L(\testg_{2l+3}\otimes \testf, l+2) =\cEul^{HR}_N(l)   \cdot L(f, \Psi_l^{-1}, 0).
\]
Moreover, for $l=-1$ we have  $\cEul^{\mathrm{HR}}_N(-1) \ne 0$.
\end{lemma}

\begin{proof} The $L$-functions $L(\testg_{2l+3}\otimes \testf,s)$ and $L(g_{2l+3}\otimes f, l+2)$ are defined by an Euler product with exactly the same local factors at all primes $q$ except possible for the primes $q\mid N$. From the definition of L-factors we have that $ \cEul^{\mathrm{HR}}_N(l)$ lies in $\Q_{\psi}(f_N)$ for all $l$.

For $l=-1$ we have $\cEul_N^{HR}(-1) = \cEul_N(1) \in \Q_{\psi}(f_N)^\times$ by Lemma \ref{proposition1}
\end{proof}

Secondly, we have the following classical formula for the Petersson product, due essentially to H. Petersson (cf.\,\cite[Theorem 5.1]{Hi81}, \cite[Satz 6]{Pe}): 

\begin{proposition}\label{prop2} Set the index $\Im(N) :=[\mathrm{SL}_2(\Z): \Gamma_0(N)] = \prod_{q^{n_q}\mid \mid N} q^{n_q-1}(q+1)$ and define
\[
\mathfrak{f}_{Pet}(l) := \frac{\Im(N)}{\Im(D_Kc^2)}\cdot \frac{(2l+2)!}{2^{4l+4} \pi^{2l+3}} \cdot \frac{h_c \cdot \sqrt{D_Kc^2}}{w_c}
\]
where $h_c$ and $w_c$ are the class number and the number of roots of unity of the order $\cO_c$ of conductor $c$. Then
\begin{equation}\label{sym-formula}
\langle g_{2l+3}^*,g_{2l+3}^*\rangle_N =   \mathfrak{f}_{Pet}(l) \cdot L(\psi^2_{2l+2},2l+3).
\end{equation}

\end{proposition}

\vspace{0.2cm}


Let us introduce now the Hecke  character  $\Phi_l = \psi_{2l+2}^{-2}\mathbf{N}^{2l+3}$ of infinity type $(2l+3, -2l-1)$, and note that $\Phi_l$ lies in the region of interpolation for the Katz $p$-adic $L$-function. Since $L(\Phi_l^{-1},s) = L(\psi^2_{2l+2},s+2l+3)$, it follows from \eqref{sym-formula} that
\begin{align*}
\langle \testg_{2l+3}^*, \testg_{2l+3}^* \rangle_{l,N} &= \cEul_N^{Pet}(l) \cdot \langle g_{2l+3}, g_{2l+3} \rangle_{l,N} \\
&= \cEul_N^{Pet}(l) \cdot \fpet (l)\cdot L(\Phi_l^{-1},0)
\end{align*}
where $\cEul_N^{Pet}(l) \in \Q_{\psi}$ is a non-zero number arising from the discrepancy at the primes $q\mid N$ of the local Hecke polynomials of $g_{2l+3}$ and $\testg_{2l+3}$.


Set
\[
	\cEul_N(l) = \frac{\cEul_N^{\mathrm{HR}}(l)}{\cE_c \cdot \cEul_N^{Pet}(l)} \qquad \text{and} \qquad \mathfrak f_\infty(l) :=  \frac{\mathfrak{f}_{\mathrm{HR}}(l) \cdot \mathfrak{f}_{K}(\Phi_l)}{\mathfrak{f}_{\mathrm{BDP}}(\Psi_l) \cdot \mathfrak{f}_{Pet}(l)}
\]
and define the function
$$
\mathfrak f: U^\circ_{\hg}  \lra \C_p, \quad   \mathfrak f(l) := \cEul_N(l) \cdot \mathfrak f_\infty(l).
$$

\begin{theorem} \label{factor} The function $\mathfrak f$ interpolates to a $p$-adic analytic function on $U_{\hg}$ and the following factorisation of $p$-adic $L$-series holds:
\begin{equation}\label{main-fact}
L_p(\htg,\testf)(2l+3) \times  L_p(K)(\Phi_l) = \mathfrak{f}(l) \cdot L_p(f,K)(\Psi_l).
\end{equation}
\end{theorem}

\begin{proof}

In our setting, the Bertolini-Darmon-Prasanna interpolation formula reads as 
\begin{align} \label{1variableBDP}
	L(f,\Psi_l^{-1},0) &= \frac{1}{\ebdp(\Psi_l)   \cdot \fbdp(\Psi_l) \cdot  \cE_c}\cdot \frac{ \Omega^{4l+4}  }{\Omega_p^{4l+4}} \cdot  L_p(f,K)(\Psi_l), \\
	\mathfrak{e}_{\mathrm{BDP}}(\Psi_l) &= (1-a_p(f) \psi_{2l+2}(\bar\wp) p^{-2-l} + \psi_{2l+2}(\bar\wp) ^2 p^{-2l-3} )^2,  \notag \\
\mathfrak{f}_{\mathrm{BDP}}(\Psi_l)  &= \left(\frac{2\pi}{c \sqrt{ D_K}}\right)^{2l+1}  l!(l+1)!  \cdot 2^{\sharp q\mid (D_K,N_E)} \cdot\omega(f,\Psi_l)^{-1}, \notag \\
\cE_c &= \prod_{q\mid c}\frac{q-\chi_K(q)}{q-1} \notag .
\end{align}
After replacing $k_1=2l+3$ and $k_2=-2l-1$ in the Katz interpolation formula we find:
\begin{align}\label{1variableKatz}
L(\Phi_l^{-1},0) &=  \frac{1}{\mathfrak{e}_{K}(\Phi_l) \cdot \mathfrak{f}_{K}(\Phi_l)} \cdot \frac{\Omega^{4l+4}}{\Omega_p^{4l+4}} \cdot L_p(K)( \Phi_l), \\
\mathfrak{e}_{K}(\Phi_l) &= (1-\psi_{2l+2}^{-2}(\wp) p^{2l+2}) (1-\frac{\psi_{2l+2}^{2}(\bar\wp)}{p^{2l+3}}),  \notag \\
\mathfrak{f}_{K}(\Phi_l) &= \left( \frac{2 \pi}{\sqrt{D_K}}\right)^{2l+1}(2l+2)!  \notag 
\end{align}
Subsequently substitute equations \eqref{sym-formula}, \eqref{1variableBDP} and \eqref{1variableKatz} into equation \eqref{1variable}. Notice that by equation \eqref{useful} we have $\ehr(l) \ekatz(l) = \ebdp(l)$. An elementary manipulation immediately shows that the decomposition formula \eqref{main-fact} holds at all the points in $U_{\hg}^\circ$.

Since $U_{\hg}^\circ$ is dense in $U_{\hg}$, in order to complete the proof of the theorem it only remains to prove the claim that  $\mathfrak f$ extends to a $p$-adic analytic function on $U_{\hg}$. 

Recall that $\mathfrak f(l)$ was defined as the product of two functions $\cEul_N(l)$ and $\mathfrak f_\infty(l)$. The Euler factors encoded in $\cEul_N(l)$ are rational functions on powers of primes $q \mid N$ and hence $\cEul_N$ extends to a $p$-adic analytic function on $U_{\hg}$. 

As for the function $\mathfrak f_\infty(l)$, we first notice that from \cite[equation (5.1.11)]{BDP1} we can compute that $\omega(f,\Psi_l) = (-1/N)^{l+1} \cdot \psi_{2l+2}(\mathfrak N)^{-1}$. Hence $\omega(f,\Psi_l)$ also interpolates to $p$-adic analytic function on $U_{\hg}$. Combining the explicit recipes given in the text for $\mathfrak{f}_{\mathrm{HR}}(l)$, $\mathfrak{f}_{K}(\Phi_l)$, $\mathfrak{f}_{\mathrm{BDP}}(\Psi_l)$ and $\mathfrak{f}_{Pet}(l)$ one readily checks that
\begin{equation}\label{foo}
\mathfrak{f}_\infty(l)  =  - \frac{\Im(D_Kc^2)}{\Im(N)}\cdot \frac{N\cdot 2^{-\sharp {q\mid (D_K,N_E)}}}{h_c\cdot D_K}  \cdot \frac{\psi_{2l+2}(\mathfrak N)}{c^{-2l} \cdot N^{l+1}}.
\end{equation}
From this explicit description the claim follows.
\end{proof}


Thanks to the above result we can already prove Theorem \ref{main-theta}. We evaluate  \eqref{main-fact} at the point of weight one in $U_{\hg}$ arising when we set $l=-1$. In this case 
Theorem \ref{factor} asserts that
\begin{equation}\label{main-fact-one}
L_p(\htg,\testf)(1) \times  L_p(K)( \psi^{-2}\mathbf{N}) = \mathfrak{f}(-1) \cdot L_p(f,K)(\psi^{-1} \mathbf{N}).
\end{equation}

Since $L_p(K)((\psi)^{-2}\mathbf{N}_K) = L_p(K)(\psi^{-2})$ (c.f. \cite[p. 90-91]{Gr}), Proposition \ref{vitasalva} combined with equations \eqref{thm:bdp-gz} and \eqref{Katz-GZ} show that
\begin{equation}\label{main1}
\int_{\testgamma}\testf \cdot E_{1,\chi_{_{K,N}}} = \lambda(\testf, \testg) \cdot \frac{\log^2_{E,p}(P_\psi)}{\log_p(u_g)}.
\end{equation}
where 	
$$
\lambda(\testf, \testg) = \cEul_N(-1) \cdot  \mathfrak f_\infty(-1) \cdot \frac{\mathfrak f_p(f,\psi)}{\mathfrak f_p(\psi)}.
$$

From equation \eqref{foo} we hence derive that
$$
\mathfrak{f}_{\infty}(-1) = -\frac{\Im(D_Kc^2)}{\Im(N)} \cdot \frac{ N}{D_Kc^2 \cdot h_c} \cdot 2^{-\sharp {q\mid (D_K,N_E)}}\cdot \psi(\mathfrak N) \in \Q_\psi(f_N)^\times.
$$
By Lemma \ref{prop1}, $\lambda(\testf, \testg) \in \Q_{\psi}(f_N)^\times$ hence the first statement of Theorem \ref{main-theta} follows.

For the second part of our main theorem, assume that $N=D_K=N_E$ and $c=1$, and put $\testf=f$, $\testg=g$ and $\testh = h = E_{1,\chi_{_{K}}}$. In this setting $\cEul_N(-1) = 1$ and $$\lambda(f,g) =  \mathfrak f_{\infty}(-1) \cdot \frac{ \mathfrak f_p(f,\psi)}{\mathfrak f_p(\psi)}.$$ Moreover it follows from genus theory that $g_K := [\Cl_K: \Cl_K^2] = 2^{\#q\mid D_K -1}$ so that $\mathfrak f_\infty(-1) = \frac{-1}{2h_Kg_K}$.
The proof follows after combining that with formulae \eqref{Katz-GZ} and \eqref{thm:bdp-gz}.



\end{document}